\documentclass{tac}

\usepackage[utf8]{inputenc}
\DeclareFontFamily{U}{min}{}

\DeclareFontShape{U}{min}{m}{n}{<-> udmj30}{}

\usepackage{mathtools}

\usepackage{amsmath}
\usepackage{extpfeil}

\usepackage{amssymb}
\usepackage{textcomp}
\usepackage{booktabs}
\usepackage{adjustbox}
\usepackage{ltablex,array}
\usepackage{longtable}
\usepackage{makecell}
\usepackage{marvosym}

\let\svthefootnote\thefootnote
\newcommand\freefootnote[1]{%
  \let\thefootnote\relax%
  \footnotetext{#1}%
  \let\thefootnote\svthefootnote%
}

\newcommand*{\faktor}[2]{%
  \raisebox{0.5\height}{\ensuremath{#1}}%
  \mkern-5mu\diagup\mkern-4mu%
  \raisebox{-0.5\height}{\ensuremath{#2}}%
} 
\newcommand\restr[2]{{%
\left.
\kern-
\nulldelimiterspace %
#1 %
\right|_{#2} %
}}
\usepackage{tikz}
\usepackage{tikz-cd}
\tikzcdset{scale cd/.style={every label/.append style={scale=#1},
    cells={nodes={scale=#1}}}}
\usepackage{quiver}
\usetikzlibrary{positioning} \usetikzlibrary{decorations.pathreplacing}
\usetikzlibrary{arrows}
\tikzset{>=stealth}

\newtheorem{question}[theorem]{Question}

\usepackage[
backend=biber,
style=alphabetic,
sorting=nyt
]{biblatex}

\begin{document}
\title{Various topos of 
types constructions}
\author{Kristóf Kanalas}
\date{}

\maketitle

\begin{abstract}
    We study and compare some topos of types constructions, which were defined by Garner, Joyal, Reyes and Makkai. \freefootnote{I would like to acknowledge the support of Masaryk University (project MUNI/A/1569/2024). I thank John Bourke for directing me towards this topic.}
\end{abstract}

\tableofcontents

\section{Introduction}

The aim of this paper is to study three "topos of types" constructions: the one from \cite{GARNER2020102831}, a variant of the one from \cite{joyaltoptp}, and the one from \cite{10.1007/BFb0090947}. 

Given a small coherent category $\mathcal{C}$, one defines the category of ultrafilters $Spec^{\neg }(\mathcal{C})$ as follows: its objects are pairs $(x,p)$ where $x\in \mathcal{C}$ and $p$ is an ultrafilter on the Boolean algebra of complemented subobjects of $x$. Its morphisms $[f]:(x,p)\to (y,q)$ are given by $q$-continuous $p$-germs of $x\to y$ maps (Definition \ref{specdef}). The category of prime filters $Spec(\mathcal{C})$ is defined analogously, except that now $p$ and $q$ are prime filters on the subobject lattices. 

The three topos of types constructions are given by three conservative topos-valued models

\[\begin{tikzcd}
	&& {\mathcal{C}} \\
	\\
	{Sh(Spec^{\neg}(\mathcal{C}),\tau _{E})} && {Sh(Spec(\mathcal{C}),\tau _{E})} && {Sh(Spec(\mathcal{C}),\tau _{sat})}
	\arrow[from=1-3, to=3-1]
	\arrow[from=1-3, to=3-3]
	\arrow[from=1-3, to=3-5]
	\arrow[from=3-1, to=3-3]
	\arrow[from=3-3, to=3-5]
\end{tikzcd}\]
each satisfying a certain extension property, explained in Theorem \ref{main1}, Theorem \ref{main2} and \cite[Theorem 1.1]{10.1007/BFb0090947}. The horizontal maps are described in Theorem \ref{allthree}. 

We will only discuss the first two versions in detail. The main results regarding the ultrafilter case $Sh(Spec^{\neg }(\mathcal{C}), \tau _{E})$ are summarized in Theorem \ref{main1}, which slightly extends \cite[Theorem 42]{GARNER2020102831} (they call it their final main result). The latter proves that precomposing with $\mathcal{C}\to Sh(Spec^{\neg }(\mathcal{C}), \tau _{E})$ yields an equivalence between the $\mathcal{E}$-valued models, assuming that the topos $\mathcal{E}$ is locally connected. We give a different proof, which additionally shows that if $\mathcal{E}$ is an arbitrary topos, this precomposition map is still fully faithful.

The paper \cite{GARNER2020102831} asks for an analogous treatment of the prime filter setting, which is done in Section 4. The main result here is Theorem \ref{main2}, describing the extension property of $\mathcal{C}\to Sh(Spec(\mathcal{C}),\tau _E)$.

Our initial motivation was to express properties of types, and as a result, stability-like properties, in terms of these associated toposes. Such applications are left for future work.

\section{Types}

This section is included to support the model-theoretic intuition behind the constructions. Its content will not be used in later sections in an essential way.

In model theory almost every important concept can be expressed in terms of types. Informally, a type $p$ of a model $M$, in variables $\vec{x}$, over a parameter set $A\subseteq M$, is the description of an imaginary $\vec{x}$-tuple with parameters from $A$. For example, one can describe an infinitesimal element in the structure $(\mathbb{R},<)$ by the (finitely consistent) set of formulas $\{0<x<\frac{1}{n} : n\in \mathbb{N}^+\}$. This set extends to a type of $(\mathbb{R},<)$ over the parameter set $\mathbb{Q}$. More formally, a type is an ultrafilter over the Boolean algebra of $A$-definable sets in $\vec{x}^M$. 

To describe properties of a theory $T\subseteq L_{\omega \omega }$ (rather than properties of a model $M$ of $T$), we can also define the $\vec{x}$-types of $T$. Informally; complete descriptions of $\vec{x}$-tuples, whose existence can not be disproved from $T$. Formally; ultrafilters on the Lindenbaum-Tarski algebra of $T$ in variables $\vec{x}$. That is, the Boolean algebra of $L$-formulas with free variables $\vec{x}$, modulo the equivalence relation $\varphi (\vec{x})\sim \psi (\vec{x})$ iff $T\vdash \varphi(\vec{x})\leftrightarrow \psi (\vec{x})$. 

Write $Def(\varphi (\vec{x})^M)$ for the Boolean algebra of definable subsets of $\varphi (\vec{x})^M$. A key observation of categorical logic is that these Boolean algebras glue together. Let $Def(M)$ be the subcategory of $\mathbf{Set}$ formed by definable sets (interpretations of formulas) and definable functions (whose graph is a definable set). If $\varphi (\vec{x})^M$ is any object in $Def(M)$ then its subobject lattice is precisely $Def(\varphi (\vec{x})^M)$.

Similarly, the Lindenbaum-Tarski algebras glue together into an abstract "category of formulas" $\mathcal{C}_T$, called the syntactic category of $T$. Any model $M$ of $T$ yields a functor $ev_M:\mathcal{C}_T\to \mathbf{Set}$ which takes a formula to its evaluation in $M$ and any elementary embedding $M\to N$ yields a natural transformation $ev_M\to ev_N$. The category of definable sets $Def(M)$ is the $ev_M$-image of $\mathcal{C}_T$.

As a final observation, note that $Def(M)$ is closed under some limits and colimits (in $\mathbf{Set}$). Indeed, a finite product of definable sets is definable, the equalizer of a pair of definable functions is definable, the image of a definable function is definable and the finite union of definable subsets is definable. It turns out that the formula defining these limits/ colimits is itself a limit/ colimit in $\mathcal{C}_T$ (preserved by $ev_M$).

These ideas lead to the following fact: given a theory $T\subseteq L_{\omega \omega}$, there is a small category $\mathcal{C}_T$ with finite limits, effective epi - mono factorization of arrows, and finite unions (this is called a coherent category), such that any subobject has a complement (this is called a Boolean coherent category), and such that $ev_{\bullet }:Mod(T)\to \mathbf{Coh}(\mathcal{C}_T,\mathbf{Set})$ is an equivalence (from the category of $T$-models and elementary embeddings to the category of $\mathbf{Set}$-valued coherent functors and natural transformations).

All of this has an analogue in positive logic. If $T\subseteq L^g_{\omega \omega }$ is a coherent theory (axioms are of the form $\forall \vec{x}(\varphi (\vec{x})\to \psi (\vec{x}))$ where $\varphi $ and $\psi $ are positive existential) then it can be encoded into a small coherent category $\mathcal{C}_T$ (the "category of positive existential formulas"), such that coherent functors $\mathcal{C}_T\to \mathbf{Set}$ and natural transformations are the same as $T$-models and homomorphisms. In this positive setting the Lindenbaum-Tarski algebras are distributive lattices ($\mathcal{C}_T$ is not Boolean) and as types we take prime filters. 

Every coherent category is the syntactic category of some coherent theory (essentially uniquely). Therefore, given a coherent category $\mathcal{C}$, we will think of its objects as positive existential formulas, arrows as provably functional positive existential formulas, and subobjects as implications between formulas sharing the same set of free variables. Indeed, if $\mathcal{C}$ is a syntactic category then subobjects can be represented by monos of the form
\[
[\varphi (\vec{x})]\xrightarrow{[\varphi (\vec{x})\wedge \vec{x}=\vec{x'}]}[\psi (\vec{x'})]
\]
where $T\vdash \varphi (\vec{x})\to \psi (\vec{x})$. Therefore we think of $Sub_{\mathcal{C}}:\mathcal{C}^{op}\to \mathbf{DLat}$ as the functor which sends an object $[\psi (\vec{x})]$ to the Lindenbaum-Tarski algebra of positive existential formulas below $\psi (\vec{x})$, and the composite
\[
S_{\mathcal{C}}:\mathcal{C}\xrightarrow{Sub_{\mathcal{C}}}\mathbf{DLat}^{op}\xrightarrow{Spec} \mathbf{Set} 
\]
(sending an object $x$ to the set of prime filters on its subobject lattice) will be called the type space functor of $\mathcal{C}$.

\begin{definition}
   Let $\mathcal{C}$ be a coherent category. 
   
   The Boolean type space functor of $\mathcal{C}$ is the composite $S^{\neg}_{\mathcal{C}}: \mathcal{C}\xrightarrow{Sub^{\neg }_{\mathcal{C}}} \mathbf{BAlg}^{op}\xrightarrow{Spec} \mathbf{Set}$. It takes an object $x$ to the set of ultrafilters on the Boolean algebra of its complemented subobjects, and takes $f:x\to y$ to $f_!:S^{\neg}_{\mathcal{C}}(x)\to S^{\neg}_{\mathcal{C}}(y)$ sending $p$ to $f_!p=\{v\in Sub_{\mathcal{C}}^{\neg }(y) : f^{-1}(v)\in p\}$.

   Similarly, the type space functor of $\mathcal{C}$ is the composite $S_{\mathcal{C}}: \mathcal{C}\xrightarrow{Sub_{\mathcal{C}}} \mathbf{DLat}^{op}\xrightarrow{Spec} \mathbf{Set}$. It takes an object $x$ to the set of prime filters on its subobject lattice and takes $f$ to $f_!$ defined as above.
\end{definition}

The next propositions are left as exercises (though the first one can be looked up in \cite{GARNER2020102831}). We include them for completeness, but they will not play a significant role in the subsequent sections.

\begin{proposition}
Let $\mathcal{C}$ be a coherent category.
\begin{itemize}
    \item[i)] $S^{\neg}_{\mathcal{C}}$ preserves finite disjoint unions (including $\emptyset $) and effective epimorphisms.
    \item[ii)] If $F:\mathcal{C}\to \mathbf{Set}$ preserves finite disjoint unions then it preserves complemented monos and preimages (pullbacks) of complemented monos.
    \item[iii)] If $F$ preserves finite disjoint unions then there is a natural transformation $tp:F\Rightarrow S_{\mathcal{C}}^{\neg}$ whose $x$-component takes $a\in Fx$ to its (Boolean) type: $tp_x(a)=\{u\in Sub_{\mathcal{C}}^{\neg}(x): a \in Fu\}$.
    \item[iv)] $tp$ is sum-cartesian (i.e.~naturality squares at complemented monos are pullbacks). Moreover $tp$ is the only sum-cartesian natural transformation from $F$ to $S_{\mathcal{C}}^{\neg}$.
    \item[v)] If $F$ is arbitrary, $G$ preserves finite disjoint unions and $\alpha :F\Rightarrow G$ is a natural transformation then $F$ preserves finite disjoint unions iff $\alpha $ is sum-cartesian.
    \item[vi)] As a consequence $FC(\mathcal{C},Set)$, the category of $\mathcal{C}\to \mathbf{Set}$ finite coproduct (i.e.~disjoint union) preserving functors is equivalent to the full subcategory of the slice $\faktor{[\mathcal{C},\mathbf{Set}]}{S_{\mathcal{C}}^{\neg}}$ spanned by the sum-cartesian maps. 
\end{itemize}
\end{proposition}

\begin{proposition}
    Let $\mathcal{C}$ be a coherent category.
\begin{itemize}
    \item[i)] $S_{\mathcal{C}}$ preserves monomorphisms, preimages (pullbacks of monos) and finite unions (including $\emptyset $). Call such functors "weakly coherent". Moreover $S_{\mathcal{C}}$ preserves effective epimorphisms.
    \item[ii)] If $F:\mathcal{C}\to \mathbf{Set}$ is weakly coherent then there is a natural transformation $tp:F\Rightarrow S_{\mathcal{C}}$ whose $x$-component takes $a\in Fx$ to its type: $tp_x(a)=\{u\in Sub_{\mathcal{C}}(x): a \in Fu\}$.
    \item[iii)] $tp$ is mono-cartesian (the naturality squares at monos are pullbacks; sometimes these are called "elementary natural transformations" or "immersions"). Moreover $tp$ is the only mono-cartesian natural transformation from $F$ to $S_{\mathcal{C}}$.
    \item[iv)] $F$ is arbitrary, $G$ is weakly coherent, $\alpha :F\Rightarrow G$ is a natural transformation. If $\alpha $ is mono-cartesian then $F$ is weakly coherent. The converse is false.
    \item[v)] If $F,G$ are weakly coherent and $\alpha :F\Rightarrow G$ is arbitrary then $tp_F \leq tp_G \circ \alpha $ (meaning: for any $x\in \mathcal{C}$ and $a\in Fx$ we have $(tp_F)_x(a)\leq (tp_G\circ \alpha )_x(a)$). Equality holds iff $\alpha $ is mono-cartesian. 
    \item[vi)] As a consequence $\mathbf{WCoh}(\mathcal{C},\mathbf{Set})$, the category of weakly coherent functors $\mathcal{C}\to \mathbf{Set}$ is equivalent to the following: objects are $F\xRightarrow{\tau _F} S_{\mathcal{C}}$ mono-cartesian natural transformations ($F$ is arbitrary), arrows are triangles ($\alpha $ is arbitrary):

\[\begin{tikzcd}
	F && G \\
	& {S_{\mathcal{C}}}
	\arrow["\alpha", Rightarrow, from=1-1, to=1-3]
	\arrow[""{name=0, anchor=center, inner sep=0}, "{\tau _F}"', Rightarrow, from=1-1, to=2-2]
	\arrow[""{name=1, anchor=center, inner sep=0}, "{\tau _G}", Rightarrow, from=1-3, to=2-2]
	\arrow["\leq", draw=none, from=0, to=1]
\end{tikzcd}\]
    
\end{itemize}
\end{proposition}

\begin{definition}
\label{specdef}
    $\mathcal{C}$ is coherent.

    The category of ultrafilters on $\mathcal{C}$ (or the Boolean spectrum of $\mathcal{C}$) is the category $Spec^{\neg }(\mathcal{C})$ whose objects are pairs $(x,p)$ (with $x\in \mathcal{C}, p\in S_{\mathcal{C}}^{\neg }(x)$), and whose morphisms $(x,p)\to (y,q)$ are given by continuous $p$-germs of $x\to y$ maps. A germ is the equivalence class of partial maps $f:x\supseteq u\to y$ (with $u\in p$) where $f\sim f'$ if there is some $u''\in p$ where both are defined and $\restr{f}{u''}=\restr{f'}{u''}$. It is continuous if $v\in q$ implies $f^{-1}(v)\in p$, i.e.~$q\leq f_{!}p$ hence $q=f_!p$. This doesn't depend on the representative of $[f]$.  

    Similarly, the category of prime filters on $\mathcal{C}$ (or the spectrum of $\mathcal{C}$) is the category $Spec(\mathcal{C})$ whose objects are pairs $(x,p)$ (with $x\in \mathcal{C}, p\in S_{\mathcal{C}}(x)$), and whose morphisms $(x,p)\to (y,q)$ are given by continuous $p$-germs of $x\to y$ maps (but note that in this case $q\leq f_!p$ does not imply equality).
\end{definition}

\begin{example}
    If $\mathcal{C}=L$ is a distributive lattice, then $Spec(L)=spec(L)^{op}$ where $spec(L)$ is the poset of prime filters on $L$ (ordered by $\subseteq $).
\end{example}

The idea from \cite{GARNER2020102831} is the following. It is well-known that the slice $\faktor{[\mathcal{C},\mathbf{Set}]}{S_{\mathcal{C}}^{\neg}}$ is equivalent to the presheaf category $[\int S_{\mathcal{C}}^{\neg}, \mathbf{Set}]$. Indeed, given $F\xRightarrow{\tau } S_{\mathcal{C}}^{\neg }$ the corresponding functor $\widetilde{F}$ takes $(x,p)$ to $\{a\in Fx : \tau _x(a)=p\}$ and conversely, given $G:\int S_{\mathcal{C}}^{\neg } \to \mathbf{Set}$ the corresponding functor $\widehat{G}$ sends $x$ to $\bigsqcup _{p\in S_{\mathcal{C}}^{\neg }} G(x,p)$ and the $x$-component of $\widehat{G}\Rightarrow S_{\mathcal{C}}^{\neg }$ sends each element of the coproduct to its index. Finally, one observes that $\tau $ is sum-cartesian iff $\widetilde{F}$ factors through $\int S_{\mathcal{C}}^{\neg } \to Spec^{\neg }(\mathcal{C})$.

From this we get an equivalence $FC(\mathcal{C},\mathbf{Set})\to [Spec^{\neg }(\mathcal{C}),\mathbf{Set}]$ which takes $F$ to $(x,p)\mapsto \{a \in Fx :tp_x(a)=p\} = \bigcap _{u\in p} F(u)$ and whose quasi-inverse takes $G$ to $x\mapsto \bigsqcup _p G(x,p)$. These formulas also make sense for $\mathcal{E}$-valued functors (where $\mathcal{E}$ is a Grothendieck topos). So now the plan is to check that these maps remain inverse equivalences, possibly under some extra assumptions on $\mathcal{E}$.

\section{The ultrafilter case}

\begin{definition}
    A Grothendieck topos $\mathcal{E}$ is locally connected if every object is the disjoint coproduct of connected (i.e.~$\sqcup $-indecomposable) objects.
\end{definition}

\begin{proposition}
    $\mathcal{E}$ is locally connected iff every object is the union of connected objects iff every object can be covered with connected objects.
\end{proposition}

\begin{proof}
    The effective epic image of a connected object is connected, so if $x\in \mathcal{E}$ can be covered with connected objects, then it is a union of connected objects. Assume $x=\bigcup _i u_i$ where each $u_i$ is connected. Take $v_i=\bigcup \{ u_j: u_j\cap u_i \neq \emptyset \}$. Then $v_i$ is connected. Indeed, if $v_i=a\sqcup b$ then $u_i=(a\cap u_i)\sqcup (b\cap u_i)$ hence it equals one of the summands, so either $u_i\subseteq a$ or $u_i\subseteq b$. Assume $u_i\subseteq a$. Similarly, for any $u_j$ s.t.~$u_j\cap u_i\neq \emptyset$ we have $u_j\subseteq a$ or $u_j\subseteq b$. As $u_i\cap u_j \neq \emptyset $ we have $u_j\subseteq a$, meaning $v_i\subseteq a$. For any $i,j$ either $v_i=v_j$ or $v_i\cap v_j =\emptyset $. We got that $x$ is a disjoint union of connected objects.
\end{proof}

\begin{proposition}
\label{FCadjunction}
    $\mathcal{C}$ is coherent, $\mathcal{E}$ is a Grothendieck topos. Then there are maps
\[\begin{tikzcd}
	{[Spec^{\neg }(\mathcal{C}),\mathcal{E}]} && {FC(\mathcal{C},\mathcal{E})} && {[Spec^{\neg }(\mathcal{C}),\mathcal{E}]}
	\arrow["{\widehat{(-)}}", from=1-1, to=1-3]
	\arrow["{\widetilde{(-)}}", from=1-3, to=1-5]
\end{tikzcd}\]
with $\widetilde{F}(x,p)= \bigcap _{u\in p}Fu$ and $\widehat{G}(x)=\bigsqcup _{p\in S_{\mathcal{C}}^{\neg }(x)}G(x,p)$, whose composite is (isomorphic to) identity. Moreover $\widehat{(-)}$ is left adjoint to $\widetilde{(-)}$ (so we see a full coreflective subcategory). When $\mathcal{E}$ is locally connected, these maps are inverse equivalences. 
\end{proposition}

\begin{proof}
    $\widetilde{F}$ is a functor: If $[f]:(x,p)\to (y,q)$ is an arrow then $q=f_!p$, so whenever $v\in q$, its preimage $f^{-1}(v)\in p$. It follows that $\restr{F(f)}{\bullet }:\bigcap _{u\in p}Fu \to \bigcap _{v\in q}Fv$ makes sense and does not depend on the representative of $[f]$.

    $\widetilde{(-)}$ is a functor: If $\alpha :F\Rightarrow F'$ is a natural transformation then $\alpha _x: Fx\to F'x$ restricted to $Fu$ lands in $F'u$ hence $\restr{\alpha _x}{\bullet }:\bigcap _{u\in p}Fu \to \bigcap _{u\in p}F'u$ makes sense. The naturality squares commute as they are restrictions of the naturality squares of $\alpha $.

    $\widehat{G}$ is a functor: Given $f:x\to y$, $\widehat{G}(f):\bigsqcup _{p\in S_{\mathcal{C}}^{\neg }(x)}G(x,p) \to \bigsqcup _{q\in S_{\mathcal{C}}^{\neg }(y)}G(y,q)$ is induced by the maps $G([f]):G(x,p)\to G(y,f_!p)$.

    $\widehat{G}$ preserves finite disjoint unions: $\emptyset \in \mathcal{C}$ is sent to an empty coproduct (since $Sub^{\neg }_{\mathcal{C}}(\emptyset )=*$ has no ultrafilters). Given $x\to x\sqcup y \leftarrow y$ its $\widehat{G}$-image is the disjoint coproduct:

\[
\adjustbox{width=\textwidth}{
\begin{tikzcd}
	{\bigsqcup _{p'\in S_{\mathcal{C}}^{\neg }(x)}G(x,p')} && {\bigsqcup _{p'\in S_{\mathcal{C}}^{\neg }(y)}G(y,p')} \\
	{\bigsqcup _{p\in [x]\subseteq S_{\mathcal{C}}^{\neg }(x\sqcup y )}G(x\sqcup y,p)} & {\bigsqcup _{p\in S_{\mathcal{C}}^{\neg }(x\sqcup y )}G(x\sqcup y,p)} & {\bigsqcup _{p\in [y]\subseteq S_{\mathcal{C}}^{\neg }(x\sqcup y )}G(x\sqcup y,p)}
	\arrow[Rightarrow, no head, from=1-1, to=2-1]
	\arrow[Rightarrow, no head, from=1-3, to=2-3]
	\arrow[from=2-1, to=2-2]
	\arrow[from=2-3, to=2-2]
\end{tikzcd}
}
\]
as $G$ preserves the isomorphism $[1_{x\sqcup y}]:(x,\restr{p}{x})\to (x\sqcup y, p)$.

$\widehat{(-)}$ is a functor: Given $\beta :G\Rightarrow G'$, the $x$-component of the corresponding natural transformation is $\bigsqcup _p G(x,p)\xrightarrow{\bigsqcup _p \beta _{(x,p)}} \bigsqcup _p G'(x,p)$.

$\widetilde{(-)}\circ \widehat{(-)}\cong id$: We need that the dashed mono in
\[\begin{tikzcd}
	{\bigcap _{u\in p}\bigsqcup _{q\in S_{\mathcal{C}}^{\neg } (u)} G(u,q)=\bigcap _{u\in p}\bigsqcup _{q\in [u]\subseteq S_{\mathcal{C}}^{\neg } (x)} G(x,q)} && {\bigsqcup _{q\in S_{\mathcal{C}}^{\neg } (x)} G(x,q)} \\
	&& {G(x,p)}
	\arrow[hook, from=1-1, to=1-3]
	\arrow[dashed, hook', from=2-3, to=1-1]
	\arrow[hook, from=2-3, to=1-3]
\end{tikzcd}\]
is epi. (Naturality is immediate.) This is clear, since the horizontal subobject is disjoint from each $G(x,p')$ when $p'\neq p$, as for some $u\in p$ we must have $p'\not \in [u]$.

$\widehat{(-)}\dashv \widetilde{(- )}$: Given a natural transformation $\widehat{G}\Rightarrow F$, restricting its $x$-component $\widehat{G}(x)=\bigsqcup _{p\in S_{\mathcal{C}}^{\neg }(x)}G(x,p) \to Fx$ to one of the $G(x,p)$'s will factor through each $Fu$ for $u\in p$, hence it factors through $\widetilde{F}(x,p)=\bigcap _p Fu$. Conversely, given a natural transformation $G\Rightarrow \widetilde{F}$ taking the disjoint coproduct of its $(x,p)$-components for $p\in  S_{\mathcal{C}}^{\neg }(x)$ yields a natural transformation $\widehat{G}\Rightarrow F$. These maps are mutual inverses, and they are natural in $F$ and in $G$.

$\widehat{(-)}\circ \widetilde{(-)}\cong id$: We need that $\bigsqcup _{p\in S_{\mathcal{C}}^{\neg}(x)} \bigcap _{u\in p} Fu  \hookrightarrow Fx$ is an epimorphism (it is mono since for $p\neq p'$ the subobjects $\bigcap _{u\in p} Fu$ and $\bigcap _{u'\in p'} Fu'$ are disjoint and the naturality squares commute as $\widehat{(-)}\circ \widetilde{(-)}$ on $\alpha :F\Rightarrow F'$ is its restriction). As now $\mathcal{E}$ is assumed to be locally connected, we can write $Fx=\bigcup a_i $ where each $a_i $ is $\sqcup $-indecomposable, hence for any complemented $u\hookrightarrow x$ either $a_i\subseteq Fu$ or $a_i\subseteq Fu^c$. Clearly $p_i=tp(a_i)=\{u: a_i\subseteq Fu \}$ is an ultrafilter and $a_i\subseteq \bigcap _{u\in p_i} Fu$.
\end{proof}

\begin{definition}
\label{flat}
    $\mathcal{C}$ is arbitrary, $\mathcal{E}$ is a Grothendieck topos. $F:\mathcal{C}\to \mathcal{E}$ is flat if for any finite diagram coming from $\mathcal{C}$, its limit (in $\mathcal{E}$) is covered by the maps induced by the cones that are coming from $\mathcal{C}$.

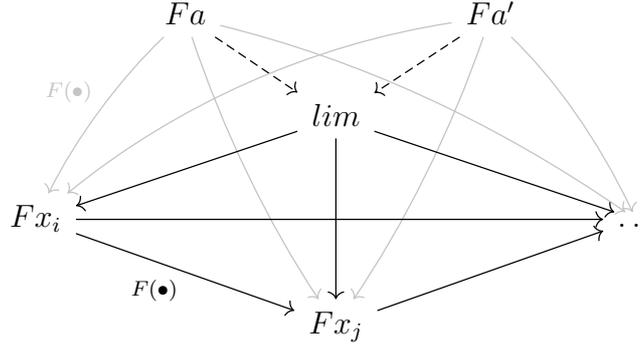
\[
\adjustbox{scale=0.85}{
\begin{tikzcd}
	& Fa && {Fa'} \\
	&& lim \\
	{Fx_i} &&&& \dots \\
	&& {Fx_j}
	\arrow[dashed, from=1-2, to=2-3]
	\arrow["{F(\bullet)}"', color={rgb,255:red,196;green,196;blue,196}, curve={height=6pt}, from=1-2, to=3-1]
	\arrow[color={rgb,255:red,196;green,196;blue,196}, curve={height=-12pt}, from=1-2, to=3-5]
	\arrow[color={rgb,255:red,196;green,196;blue,196}, curve={height=6pt}, from=1-2, to=4-3]
	\arrow[dashed, from=1-4, to=2-3]
	\arrow[color={rgb,255:red,196;green,196;blue,196}, curve={height=18pt}, from=1-4, to=3-1]
	\arrow[color={rgb,255:red,196;green,196;blue,196}, curve={height=-6pt}, from=1-4, to=3-5]
	\arrow[color={rgb,255:red,196;green,196;blue,196}, curve={height=-6pt}, from=1-4, to=4-3]
	\arrow[from=2-3, to=3-1]
	\arrow[from=2-3, to=3-5]
	\arrow[from=2-3, to=4-3]
	\arrow[from=3-1, to=3-5]
	\arrow["{F(\bullet)}"', from=3-1, to=4-3]
	\arrow[from=4-3, to=3-5]
\end{tikzcd}
}
\]
\end{definition}

For topos-valued flat functors a possible reference is \cite[Sections VII.8 and VII.9]{sheaves}.

\begin{remark}
When $\mathcal{E}=\mathbf{Set}$ this expresses that $(\int F)^{op}$ is filtered.
\end{remark}

\begin{proposition}
    If $\mathcal{C}$ has finite limits then $F$ is flat iff it is lex.
\end{proposition}

\begin{proposition}
    In Definition \ref{flat} it suffices to check the following diagrams: the empty diagram, a pair of objects and a parallel pair of arrows.
\end{proposition}

\begin{proof}
    First note that we get the condition for joint equalizers. Indeed, let $f_i,g_i:x\to y_i$ $(i=1,\dots n)$ be a finite set of parallel pairs. Let $i_1:e_1\hookrightarrow Fx$ be the equalizer of the first pair $(Ff_1,Fg_1)$. There is a cover $(Fa_i\to e_1)_i$ such that each composite $Fa_i\to e_1\hookrightarrow Fx$ is coming from $\mathcal{C}$, say, it is $Fh_i$. 

    Let $i_2:e_2\hookrightarrow e_1$ be the equalizer of $(Ff_2 \circ i_1, Fg_2\circ i_1)$. Then $e_2\hookrightarrow e_1\hookrightarrow Fx$ is the joint equalizer of the first two pairs. It is easy to check that the pullback of $i_2$ along $Fa_i\to e_1$, say $e_{2,i}\hookrightarrow Fa_i$, is the equalizer of $(Ff_2\circ Fh_i, Fg_2\circ Fh_i)$, hence it can be covered with a family $(Fb_{i,j}\to e_{2,i})_j$, such that each composite $F{b_{i,j}}\to e_{2,i}\to Fa_i$ is coming from $\mathcal{C}$, say, it is $Fk_{i,j}$.

\[
\adjustbox{scale=0.85}{
\begin{tikzcd}
	{Fb_{i,j}} \\
	\\
	{e_{2,i}} && {Fa_i} \\
	& pb &&&&& {Fy_1} \\
	{e_2} && {e_1} && Fx \\
	&&&&&& {Fy_2}
	\arrow[from=1-1, to=3-1]
	\arrow["{Fk_{i,j}}"{description}, from=1-1, to=3-3]
	\arrow[hook, from=3-1, to=3-3]
	\arrow[from=3-1, to=5-1]
	\arrow[from=3-3, to=5-3]
	\arrow["{Fh_i}"{description}, from=3-3, to=5-5]
	\arrow[hook, from=5-1, to=5-3]
	\arrow[hook, from=5-3, to=5-5]
	\arrow["{Fg_1}"{description, pos=0.7}, shift right=2, from=5-5, to=4-7]
	\arrow["{Ff_1}"{description}, shift left=2, from=5-5, to=4-7]
	\arrow["{Fg_2}"{description, pos=0.7}, shift right=2, from=5-5, to=6-7]
	\arrow["{Ff_2}"{description}, shift left=2, from=5-5, to=6-7]
\end{tikzcd}
}
\]

As the pullback of a covering family is a covering family, $(e_{2,i}\to e_2)_i$ is a cover. As the composite of a height 2 tree of covering families is a cover, we get that $(Fb_{i,j}\to e_{2,i}\to e_2)_{i,j}$ is a cover. By induction we get that the condition holds for finite joint equalizers.

In a similar fashion we get that the condition holds for finite products. Given an arbitrary diagram, its limit can be computed as the joint equalizer of the pairs 
\[
\adjustbox{scale=0.85}{
\begin{tikzcd}
	{\prod _i Fx_i} &&&& {Fx_{i_1}} \\
	&& {Fx_{i_0}}
	\arrow["{\pi _{i_1}}", from=1-1, to=1-5]
	\arrow["{\pi _{i_0}}"', from=1-1, to=2-3]
	\arrow["Ff"', from=2-3, to=1-5]
\end{tikzcd}
}
\]
where $Ff$ is an arrow from the diagram (and $\prod _i Fx_i$ is the product of the objects in the diagram). There is a cover $(\prod Fx_i \leftarrow Fa_j)_j$ such that each composite $Fa_j\to \prod Fx_i \xrightarrow{\pi _{i_0}}Fx_{i_0}$ is coming from $\mathcal{C}$, say, it is $Fh_{j,i_0}$. In
\[
\adjustbox{scale=0.85}{
\begin{tikzcd}
	&& lim && {e_j} && {Fb_{j,t}} \\
	&&& pb \\
	&& {\prod _i Fx_i} && {Fa_j} \\
	& {Fx_{i_0}} \\
	{Fx_{i_1}}
	\arrow[hook', from=1-3, to=3-3]
	\arrow[from=1-5, to=1-3]
	\arrow[hook', from=1-5, to=3-5]
	\arrow[from=1-7, to=1-5]
	\arrow["{Fk_{t,j}}"{description}, from=1-7, to=3-5]
	\arrow["{\pi _{i_0}}"', shift right, from=3-3, to=4-2]
	\arrow["{\pi _{i_1}}", shift left=4, from=3-3, to=5-1]
	\arrow[from=3-5, to=3-3]
	\arrow["{Fh_{j,i_0}}"{description}, curve={height=-6pt}, from=3-5, to=4-2]
	\arrow["{Fh_{j,i_1}}"{description}, shift left, curve={height=-24pt}, from=3-5, to=5-1]
	\arrow["Ff"', shift right, from=4-2, to=5-1]
\end{tikzcd}
}
\]
it is easy to check that the pullback of $lim$ (a joint equalizer) is the joint equalizer of the pairs $(Ff\circ Fh_{j,i_0}, Fh_{j,i_1})$. Therefore it can be covered as $(e_j\leftarrow Fb_{j,t})_t$ such that each composite $Fb_{j,t}\to e_j\to Fa_j$ is coming from $\mathcal{C}$. Then $(lim \leftarrow e_j \leftarrow Fb_{j,t})_{j,t}$ is the cover we were looking for.
\end{proof}

\begin{proposition}
\label{restrictstoflat}
    If $F$ is lex then $\widetilde{F}$ is flat, assuming $\mathcal{E}$ is locally connected. If $G$ is flat then $\widehat{G}$ is flat (hence lex), for arbitrary $\mathcal{E}$
\end{proposition}

\begin{proof}
    If $F$ is lex then $\widetilde{F}$ is flat:
    \begin{itemize}
        \item[--] the empty diagram: Since $\mathcal{E}$ is locally connected we can write $1=\bigcup a_i$ where each $a_i$ is connected. Then the maps $(\widetilde{F}(1,tp(a_i))=\bigcap _{u\in tp(a_i)}Fu\hookrightarrow 1)_i$ form a cover. ($tp(a_i)$ is an ultrafilter on $Sub^{\neg }_{\mathcal{C}}(1)$ because $F1=1$.)
        \item[--] a pair of objects: Write $\widetilde{F}(x,p)\times \widetilde{F}(y,q)=\bigcup a_i$ where each $a_i$ is connected. Then the maps $(\langle \widetilde{F}([\pi _1]),\widetilde{F}([\pi _2]) \rangle :\widetilde{F}(x\times y,tp(a_i))\hookrightarrow \widetilde{F}(x,p)\times \widetilde{F}(y,q))_i$ form a cover since $a_i\subseteq \bigcap _{w\in tp(a_i)}Fw =\widetilde{F}(x\times y,tp(a_i))$. (Again, $tp(a_i)$ is an ultrafilter on $Sub^{\neg }_{\mathcal{C}}(x\times y)$ because $F(x\times y)=Fx\times Fy$. $[\pi _1]:(x\times y,tp(a_i))\to (x,p)$ is continuous as if $u_0\in p$ then $a_i\subseteq (\bigcap _p Fu)\times (\bigcap _q Fv)\subseteq Fu_0 \times Fy =F(u_0\times y)$.)
        \item[--] a parallel pair of arrows: Similarly as before we get
\[
\adjustbox{scale=0.9}{
\begin{tikzcd}
	{\widetilde{F}(eq(f,g),tp(a_i))} \\
	\dots & {eq = \bigcup a_i} && {\widetilde{F}(x,p)} && {\widetilde{F}(y,q)}
	\arrow[dashed, from=1-1, to=2-2]
	\arrow["{\widetilde{F}([j])}", curve={height=-6pt}, from=1-1, to=2-4]
	\arrow[hook, from=2-2, to=2-4]
	\arrow["{\widetilde{F}([g])}"', shift right=2, from=2-4, to=2-6]
	\arrow["{\widetilde{F}([f])}", shift left=2, from=2-4, to=2-6]
\end{tikzcd}
}
\]
where $j:eq(f,g)\hookrightarrow u_0\hookrightarrow x$ is the inclusion of the equalizer of $f$ and $g$ assuming both are defined on $u_0\in p$. It is continuous as $a_i\subseteq eq(Ff,Fg)=F(eq(f,g))\subseteq Fu_0\subseteq Fx$, moreover $a_i\subseteq Fu$ for any $u\in p$ hence $u\in p\Rightarrow u\cap eq(f,g) \in tp(a_i)$. They jointly cover as $a_i\subseteq \widetilde{F}(eq(f,g),tp(a_i))$. 
    \end{itemize}

If $G$ is flat then $\widehat{G}$ is flat:
\begin{itemize}
    \item[--] the empty diagram: $\widehat{G}(1)=\bigsqcup _p G(1,p)\to 1$ is epic as $G$ is flat (and it is enough to consider maps of the form $G(1,p)\to 1$ as any other $G(y,q)\to 1$ factors through one of them). 
    \item[--] a pair of objects: 
    \[\widehat{G}(x\times y)\to \widehat{G}(x)\times \widehat{G}(y) = (\bigsqcup _p G(x,p))\times (\bigsqcup _q G(y,q))=\bigsqcup _{p,q} G(x,p)\times G(y,q)\]
    is epi as 
    \[\bigsqcup _{r: (\pi _1)_!r=p, (\pi _2)_!r=q} G(x\times y,r)\to G(x,p)\times G(y,q)\]
    is epi by the flatness of $G$. 
    \item[--] a parallel pair of arrows: the induced map 
    \[
    \widehat{G}(eq(f,g))\to eq(\widehat{G}(f),\widehat{G}(g))=\bigsqcup _{p: f_!p=g_!p} eq(G([f]),G([g]))
    \]
    is epi as 
    \[
    \bigsqcup _{r:\  u\in p \ \Rightarrow \ eq(f,g) \cap u\in r} G(eq(f,g),r)\to eq(G([f]),G([g]))\]
    is epi by the flatness of $G$.

\end{itemize}
\end{proof}

\begin{definition}
    We will write $E=E^{\neg}$ for the set of families
\[
\left[
\begin{tikzcd}
	{(x,p_0)} & {\dots } & {(x,p_i)} & {\dots } \\
	&& {(y,q)}
	\arrow["{[f]}"', from=1-1, to=2-3]
	\arrow["{[f]}", from=1-3, to=2-3]
\end{tikzcd}
\right]_{p_i:\ f_!p_i=q}
\]
 where $f:x\to y$ is an effective epimorphism. We write $\tau _E=\tau _{E^{\neg}}$ for the generated Grothendieck topology.   
\end{definition}

Let us give an explicit description for $\tau _E$, which will be used later, in Section 5. (This is more-or-less \cite[Proposition 41]{GARNER2020102831}.)

\begin{proposition}
\label{Ebar}
    $\tau _E$ has a basis, whose elements are given by rooted cotrees, where each branch is finite, and each vertex is either a leaf, or its predecessors form an $E$-family, or it has a single predecessor from which the map is an isomorphism:

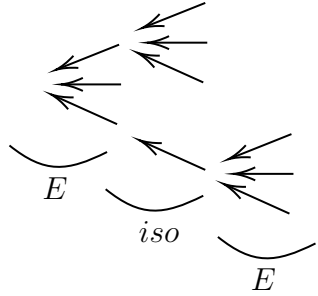
\tikzset{every picture/.style={line width=0.75pt}} %
\[
\begin{tikzpicture}[x=0.75pt,y=0.75pt,yscale=-1,xscale=1]
\draw    (170.69,123.28) -- (140.86,135.26) ;
\draw [shift={(139,136)}, rotate = 338.14] [color={rgb, 255:red, 0; green, 0; blue, 0 }  ][line width=0.75]    (10.93,-3.29) .. controls (6.95,-1.4) and (3.31,-0.3) .. (0,0) .. controls (3.31,0.3) and (6.95,1.4) .. (10.93,3.29)   ;
\draw    (171.69,143.28) -- (143.69,143.28) ;
\draw [shift={(141.69,143.28)}, rotate = 360] [color={rgb, 255:red, 0; green, 0; blue, 0 }  ][line width=0.75]    (10.93,-3.29) .. controls (6.95,-1.4) and (3.31,-0.3) .. (0,0) .. controls (3.31,0.3) and (6.95,1.4) .. (10.93,3.29)   ;
\draw    (169.69,163.28) -- (138.51,149.11) ;
\draw [shift={(136.69,148.28)}, rotate = 24.44] [color={rgb, 255:red, 0; green, 0; blue, 0 }  ][line width=0.75]    (10.93,-3.29) .. controls (6.95,-1.4) and (3.31,-0.3) .. (0,0) .. controls (3.31,0.3) and (6.95,1.4) .. (10.93,3.29)   ;
\draw    (213.69,102.28) -- (183.86,114.26) ;
\draw [shift={(182,115)}, rotate = 338.14] [color={rgb, 255:red, 0; green, 0; blue, 0 }  ][line width=0.75]    (10.93,-3.29) .. controls (6.95,-1.4) and (3.31,-0.3) .. (0,0) .. controls (3.31,0.3) and (6.95,1.4) .. (10.93,3.29)   ;
\draw    (214.69,122.28) -- (186.69,122.28) ;
\draw [shift={(184.69,122.28)}, rotate = 360] [color={rgb, 255:red, 0; green, 0; blue, 0 }  ][line width=0.75]    (10.93,-3.29) .. controls (6.95,-1.4) and (3.31,-0.3) .. (0,0) .. controls (3.31,0.3) and (6.95,1.4) .. (10.93,3.29)   ;
\draw    (212.69,142.28) -- (181.51,128.11) ;
\draw [shift={(179.69,127.28)}, rotate = 24.44] [color={rgb, 255:red, 0; green, 0; blue, 0 }  ][line width=0.75]    (10.93,-3.29) .. controls (6.95,-1.4) and (3.31,-0.3) .. (0,0) .. controls (3.31,0.3) and (6.95,1.4) .. (10.93,3.29)   ;
\draw    (213.69,186.28) -- (182.51,172.11) ;
\draw [shift={(180.69,171.28)}, rotate = 24.44] [color={rgb, 255:red, 0; green, 0; blue, 0 }  ][line width=0.75]    (10.93,-3.29) .. controls (6.95,-1.4) and (3.31,-0.3) .. (0,0) .. controls (3.31,0.3) and (6.95,1.4) .. (10.93,3.29)   ;
\draw    (256.69,168.28) -- (226.86,180.26) ;
\draw [shift={(225,181)}, rotate = 338.14] [color={rgb, 255:red, 0; green, 0; blue, 0 }  ][line width=0.75]    (10.93,-3.29) .. controls (6.95,-1.4) and (3.31,-0.3) .. (0,0) .. controls (3.31,0.3) and (6.95,1.4) .. (10.93,3.29)   ;
\draw    (257.69,188.28) -- (229.69,188.28) ;
\draw [shift={(227.69,188.28)}, rotate = 360] [color={rgb, 255:red, 0; green, 0; blue, 0 }  ][line width=0.75]    (10.93,-3.29) .. controls (6.95,-1.4) and (3.31,-0.3) .. (0,0) .. controls (3.31,0.3) and (6.95,1.4) .. (10.93,3.29)   ;
\draw    (255.69,208.28) -- (224.51,194.11) ;
\draw [shift={(222.69,193.28)}, rotate = 24.44] [color={rgb, 255:red, 0; green, 0; blue, 0 }  ][line width=0.75]    (10.93,-3.29) .. controls (6.95,-1.4) and (3.31,-0.3) .. (0,0) .. controls (3.31,0.3) and (6.95,1.4) .. (10.93,3.29)   ;
\draw    (220,220) .. controls (237.69,234.28) and (248.69,233.28) .. (268.69,223.28) ;
\draw    (116,174) .. controls (133.69,188.28) and (144.69,187.28) .. (164.69,177.28) ;
\draw    (164,196) .. controls (181.69,210.28) and (192.69,209.28) .. (212.69,199.28) ;

\draw (235,234.4) node [anchor=north west][inner sep=0.75pt]    {$E$};
\draw (131,188.4) node [anchor=north west][inner sep=0.75pt]    {$E$};
\draw (179,210.4) node [anchor=north west][inner sep=0.75pt]    {$iso$};

\end{tikzpicture}
\]
    
\end{proposition}

\begin{proof}
    First we show that this collection of families is indeed the basis of a topology (see \cite[Definition III.2.2 and Exercise III.3]{sheaves}).
    \begin{itemize}
        \item[i)] It contains all isomorphisms: clear.
        \item[ii)] It is "pullback"-stable. If $((x,p_i)\xrightarrow{[f]}(y,q))_i$ is an $E$-family, and $(y',q')\xrightarrow{[g]}(y,q)$ is arbitrary, then the family of arrows with codomain $(y',q')$, whose postcomposition with $[g]$ factors through a leg of the cover, contains a covering family. Indeed, assume that $g$ is a partial map $g:y'\supseteq v\to y$ with $v\in q'$ and write $f'$ for the pullback of $f$ along $g$. Then $f':x'\to v$ is effective epi, and whenever $p'\in S_{\mathcal{C}}^{\neg }(x')$ satisfies $f'_!p'=q'$, we get a square
\[\begin{tikzcd}
	{(x',p')} && {(x,g'_!p')} \\
	{(v,q')} \\
	{(y',q')} && {(y,q)}
	\arrow["{[g']}", from=1-1, to=1-3]
	\arrow["{[f']}"', from=1-1, to=2-1]
	\arrow["{[f]}", from=1-3, to=3-3]
	\arrow["{[id]}"', from=2-1, to=3-1]
	\arrow["{[g]}"{description}, from=2-1, to=3-3]
	\arrow["{[g]}"', from=3-1, to=3-3]
\end{tikzcd}\]
that is, $f$ is continuous wrt.~$g'_!p'$, as $f_!g'_!p'=g_!f'_!p'=g_!q'=q$. So the "pullback" family contains a cover of the form: an $E$-family post-composed with an isomorphism. Given a composite cotree as above, we can "pull it back" levelwise. 
        \item[iii)] Closed under building cotrees of height 2: clear.
    \end{itemize}
Finally note that any $Spec^{\neg}(\mathcal{C})\to \mathcal{E}$ functor which sends the $E$-families to epimorphic ones, does the same with the $\tau _E$-families (in the terminology of \cite{presheaftype}: every topos is an $\omega $-topos, see \cite[Remark 2.9]{presheaftype}). So if $\tau $ is a Grothendieck topology with $E\subseteq \tau \subseteq \tau _E$ then $\#Y:Spec^{\neg }(\mathcal{C})\to Sh(Spec^{\neg }(\mathcal{C}),\tau )$ turns the $\tau _E$-families epimorphic, hence $\tau _E\subseteq \tau $ by \cite[Proposition 1.3.3(2)]{makkai}. It follows that $\tau _E$ is the smallest topology containing $E$.
    
\end{proof}

\begin{proposition}
    If $F$ preserves effective epis, then $\widetilde{F}$ is $E$-preserving, assuming $\mathcal{E}$ is locally connected. If $G$ is $E$-preserving, then $\widehat{G}$ preserves effective epis, for arbitrary $\mathcal{E}$.
\end{proposition}

\begin{proof}
    First claim: Let $f:x\to y$ be an effective epi in $\mathcal{C}$. Then $Ff:Fx\to Fy$ is effective epi. Since $\mathcal{E}$ is locally connected we can write $Fx=\bigcup a_i$ where each $a_i$ is connected. Let $b_i$ be the image of $a_i$ under $Ff$. Then $b_i$ is connected and they together cover $Fy$. Since each $b_i$ has a type it follows that $\widetilde{F}(y,q)= \bigcap _{v\in q} Fv$ is covered by those $b_i$'s whose type is $q$ (the rest is disjoint). Since any such $b_i$ is covered by $a_i\subseteq \bigcap _{u\in tp(a_i)}Fu=\widetilde{F}(x,tp(a_i))\xrightarrow{\widetilde{F}([f])=\restr{Ff}{\bullet }}\widetilde{F}(y,q)$ we are done.

     Second claim: We need that $\bigsqcup _{p\in S_{\mathcal{C}}^{\neg }(x)}G(x,p)\to \bigsqcup _{q\in S_{\mathcal{C}}^{\neg }(y)}G(y,q)$ is effective epi, i.e.~that for any $q$ the family $(G([f]):G(x,p_i)\to G(y,q))_{f_{!}p_i =q} $ is epimorphic. This is precisely our assumption.
\end{proof}

\begin{theorem}
\label{main1}
    Let $\mathcal{C}$ be either a coherent category with finite disjoint coproducts or a Boolean coherent category. Write $\varphi =\varphi ^{\neg}:\mathcal{C}\to Sh(Spec^{\neg }(\mathcal{C}),\tau _E)$ for the functor taking $y$ to $Spec^{\neg }(\mathcal{C})(-,y)^{\#}$ (sheafification of the presheaf which sends $(x,p)$ to the set of $p$-germs of $x\to y$ maps). Then:
    \begin{itemize}
        \item[i)] $\varphi $ is coherent.
        \item[ii)] Given any Grothendieck topos $\mathcal{E}$, $ \varphi ^{\circ}:Fun^*(Sh(Spec^{\neg }(\mathcal{C})),\mathcal{E})\to \mathbf{Coh}(\mathcal{C},\mathcal{E})$ is fully faithful.
        \item[iii)] When $\mathcal{E}$ is locally connected $ \varphi ^{\circ}$ is an equivalence.
    \end{itemize}
\end{theorem}

\begin{proof}
    By the previous propositions, given any Grothendieck topos $\mathcal{E}$ we have 
    \[
    E-Flat(Spec^{\neg}(\mathcal{C}),\mathcal{E})\xrightarrow{\widehat{(-)}}\mathbf{Coh}(\mathcal{C},\mathcal{E})\xrightarrow{\widetilde{(-)}}[Spec^{\neg}(\mathcal{C}),\mathcal{E}]
    \]
    whose composite is naturally isomorphic to the full embedding (and the second map lands in $E-Flat(Spec^{\neg}(\mathcal{C}),\mathcal{E})$ and provides a quasi-inverse, assuming $\mathcal{E}$ is locally connected). Here we need the assumption on $\mathcal{C}$ to ensure that finite disjoint union preserving regular functors are coherent.

    The map $(\#Y)^{\circ} :Fun^*(Sh(Spec^{\neg}(\mathcal{C})),\mathcal{E})\to E-Flat(Spec^{\neg}(\mathcal{C}),\mathcal{E})$ is an equivalence whose quasi-inverse is $Lan_{\#Y}$ (this is well-known, see e.g.~\cite[Theorem 3.15]{presheaftype} for a proof). So the composite 
    \[
    Fun^*(Sh(Spec^{\neg}(\mathcal{C})),\mathcal{E})\xrightarrow{(\#Y)^{\circ}} E-Flat(Spec^{\neg}(\mathcal{C}),\mathcal{E}) \xrightarrow{\widehat{(-)}}\mathbf{Coh}(\mathcal{C},\mathcal{E})
    \]
    is fully faithful, and when $\mathcal{E}$ is locally connected it is an equivalence. It remains to identify this map with $\varphi ^{\circ }$ (then in particular it follows that $\varphi $, as the image of $1_{Sh(Spec^{\neg }(\mathcal{C}))}$, is coherent; although this can be checked directly as well).

    So let $M^*:Sh(Spec^{\neg }(\mathcal{C}))\to \mathcal{E}$ be lex cocontinuous and $f:x\to y$ be a morphism in $\mathcal{C}$. We have
\[
\adjustbox{width=\textwidth}{
\begin{tikzcd}
	{\bigsqcup _{p\in S_{\mathcal{C}}^{\neg}(x)}(M^*\circ \# \circ Y)(x,p)} && {M^*\circ \#\ (\bigsqcup _{p} Y(x,p))} & {M^*\circ \# \ (Spec^{\neg }(\mathcal{C})(-,x))} \\
	\\
	{\bigsqcup _{q\in S_{\mathcal{C}}^{\neg}(y)}(M^*\circ \# \circ Y)(y,q)} && {M^*\circ \# \ (\bigsqcup _{q} Y(y,q))} & {M^*\circ \# \ (Spec^{\neg }(\mathcal{C})(-,y))}
	\arrow["{\cong }", from=1-1, to=1-3]
	\arrow["{\left(M^*\circ \# \circ Y \ ([f]) \right)_p}", from=1-1, to=3-1]
	\arrow[Rightarrow, no head, from=1-3, to=1-4]
	\arrow[from=1-3, to=3-3]
	\arrow["{M^*\circ \# \ (f_{\circ })}", from=1-4, to=3-4]
	\arrow["{\cong }", from=3-1, to=3-3]
	\arrow[Rightarrow, no head, from=3-3, to=3-4]
\end{tikzcd}
}
\]
This isomorphism is natural in $M^*$ which completes the proof.

\end{proof}

\begin{remark}
\label{varphigeomsurj}
    The induced map $\varphi ^*: Sh(\mathcal{C})\leftrightarrow Sh(Spec^{\neg }(\mathcal{C})):\varphi _*$ is a geometric surjection. Indeed, since $Sh(\mathcal{C})$ has enough points and $\mathbf{Set}^I$ is locally connected we get an extension
\[\begin{tikzcd}
	{Sh(\mathcal{C})} && {\mathbf{Set}^I} \\
	{Sh(Spec^{\neg }(\mathcal{C}))}
	\arrow["{\langle M^*_i\rangle _i}", from=1-1, to=1-3]
	\arrow["{\varphi ^*}"', from=1-1, to=2-1]
	\arrow[dashed, from=2-1, to=1-3]
\end{tikzcd}\]
showing that $\varphi ^*$ is conservative.
\end{remark}

\section{The prime filter case}

There will be two essential differences, compared to the previous section.

First, unlike in the ultrafilter case, $\widetilde{(-)}:\mathbf{WCoh}(\mathcal{C},\mathbf{Set})\to [Spec(\mathcal{C}),\mathbf{Set}]$ defined by $\widetilde{F}(x,p)=\{a\in Fx: tp(a)\geq p \}=\bigcap _{u\in p}Fu$ is not essentially surjective. Morally, this is because the fibers are now posets, and while $\widetilde{F}$ satisfies some fiberwise $\infty $-flatness (e.g.~given $p\leq q$, the map $\widetilde{F}(x,q)\to \widetilde{F}(x,p)$ is mono, and given a $\lambda $-indexed chain $p_0\leq p_1\leq \dots $ with union $p_{\lambda }$, $\widetilde{F}(x,p_0)\hookleftarrow \widetilde{F}(x,p_1) \hookleftarrow \dots \widetilde{F}(x,p_{\lambda })$ is an intersection), arbitrary functors need not to.

Second, once we restrict $\widetilde{(-)}$ between some suitable subcategories, its proposed quasi-inverse $\widehat{(-)}$ sends $G:Spec(\mathcal{C})\to \mathbf{Set}$ to $\widehat{G}$, defined by $x\mapsto colim_{p\in S_{\mathcal{C}}^{op}} G(x,p)$. Working with (non-filtered) colimits (of monos) is harder than dealing with disjoint coproducts, so our proofs will differ from the ones in the previous section.

\begin{proposition}
   $\mathcal{C}$ is coherent. Let us write $\varphi _0:\mathcal{C}\to \mathbf{Set}^{Spec(\mathcal{C})^{op}}$ for the functor which sends $y$ to the presheaf $Spec(\mathcal{C})(-,y)$ (taking $(x,p)$ to the set of $p$-germs of $x\to y$ maps). Then $\varphi _0$ is finite union preserving and lex.
\end{proposition}

\begin{proof}
    Terminal object: for any $(x,p)$, the set of $p$-germs of $x\to 1$ maps is a singleton.

    Binary products: $Spec(\mathcal{C})((x,p),y_1\times y_2)\to Spec(\mathcal{C})((x,p),y_1)\times Spec(\mathcal{C})((x,p),y_2)$ sending $[\langle f_1,f_2\rangle :x\supseteq u\to y_1\times y_2]$ to $([f_1],[f_2])$ is (the $(x,p)$-component of) a natural isomorphism.

    Equalizers: if $[x\supseteq u_0\xrightarrow{h}y\xrightarrow{f}z]=[x\supseteq u_0\xrightarrow{h}y\xrightarrow{g}z]$ then there's $u_0\supseteq u_1\in p$ with $f\circ \restr{h}{u_1}=g\circ \restr{h}{u_1}$, hence $\restr{h}{u_1}$ factors through the equalizer, which means $[h]\in Spec(\mathcal{C})((x,p),eq(f,g))$.

    Finite unions: $\emptyset $ is preserved, as if $(x,p)\in Spec(\mathcal{C})$ then $x\neq \emptyset $, so there are no $x\to \emptyset $ arrows. If $y=v_1\vee v_2$ then given a germ $[h:x\supseteq u_0\to y]$, we get $u_0=h^{-1}(v_1)\vee h^{-1}(v_2)$, so one of them, say $h^{-1}(v_1)$ is in $p$, and therefore $[h]\in Spec(\mathcal{C})((x,p),v_1)\subseteq Spec(\mathcal{C})((x,p),y)$.

\end{proof}

\begin{definition}
    A Grothendieck topos is prime-generated if every object is the union of some of its $\bigvee $-indecomposable subobjects. We will also refer to such (sub)objects as primes.
\end{definition}

\begin{definition}
    A functor $G:Spec(\mathcal{C})\to \mathcal{E}$ is $p_{\infty }$-flat if it is flat and given a $\lambda $-indexed chain $(p_0\leq p_1\leq \dots p_{\alpha }\leq \dots )_{\alpha <\lambda }$ of prime filters on $Sub_{\mathcal{C}}(x)$, with union $p_{\lambda }$, it follows that $G(x,p_0)\hookleftarrow G(x,p_1)\hookleftarrow \dots G(x,p_{\alpha })\hookleftarrow \dots G(x,p_{\lambda })$ is a limit (intersection) in $\mathcal{E}$.

    (Note that $G(x,p_1)\to G(x,p_0)$ is mono, as $[id]:(x,p_1)\to (x,p_0)$ is mono and flat functors preserve monomorphisms.)
\end{definition}

\begin{proposition}
\label{lexvequiv}
    $\mathcal{C}$ is coherent, $\mathcal{E}$ is a prime-generated Grothendieck topos. Then there is an equivalence of categories
\[\begin{tikzcd}
	{Flat_{p_{\infty}}(Spec(\mathcal{C}),\mathcal{E})} && {\mathbf{Lex}_{\vee}(\mathcal{C},\mathcal{E})}
	\arrow["{\widehat{(-)}}", curve={height=-12pt}, from=1-1, to=1-3]
	\arrow["{\widetilde{(-)}}", curve={height=-12pt}, from=1-3, to=1-1]
\end{tikzcd}\]
with $\widetilde{F}(x,p)= \bigcap _{u\in p}Fu$ and $\widehat{G}(x)=colim _{p\in S_{\mathcal{C}}(x)^{op}}G(x,p)$.   
\end{proposition}

\begin{proof}
    $\widetilde{F}$ is a functor: same as in the proof of Proposition \ref{FCadjunction}.

    $\widetilde{F}$ is flat: same as in the proof of Proposition \ref{restrictstoflat}, except that the $a_i$'s are now primes, i.e.~$\bigvee $-indecomposable subobjects (for this argument $\vee $-indecomposable would be enough). $\widetilde{F}$ is $p_{\infty }$-flat: we need $\bigcap _{i<\lambda } \bigcap _{u\in p_i} Fu = \bigcap _{u\in p_{\lambda }} Fu$ which is clear.
    
    $\widetilde{(-)}$ is a functor: same as in the proof of Proposition \ref{FCadjunction}.
    
    $\widehat{G}$ is a functor, and it is naturally isomorphic to $ \mathcal{C}\xrightarrow{\varphi _0} \mathbf{Set}^{Spec(\mathcal{C})^{op}}\xrightarrow{Lan_YG}\mathcal{E}$. Indeed, since $Lan_YG$ preserves all colimits, we have 
\[
\adjustbox{width=\textwidth}{
\begin{tikzcd}
	{Lan_YG(Spec(C)(-,x))} & {Lan_YG(colim _{p\in S_{\mathcal{C}}(x)^{op}}Spec(C)(-,(x,p)))} & {colim _{p\in S_{\mathcal{C}}(x)^{op}} G(x,p)} \\
	{Lan_YG(Spec(C)(-,y))} & {Lan_YG(colim _{q\in S_{\mathcal{C}}(y)^{op}}Spec(C)(-,(y,q)))} & {colim _{q\in S_{\mathcal{C}}(y)^{op}} G(y,q)}
	\arrow["\cong"{description}, draw=none, from=1-1, to=1-2]
	\arrow["{Lan_YG(f_{\circ })}"', from=1-1, to=2-1]
	\arrow["\cong"{description}, draw=none, from=1-2, to=1-3]
	\arrow["{Lan_YG(([f]_{\circ })_{(x,p)})}", from=1-2, to=2-2]
	\arrow["{(G([f]):G(x,p)\to G(y,f_!p))_{(x,p)}}", from=1-3, to=2-3]
	\arrow["\cong"{description}, draw=none, from=2-1, to=2-2]
	\arrow["\cong"{description}, draw=none, from=2-2, to=2-3]
\end{tikzcd}
}
\]

    $\widehat{G}$ preserves finite limits and finite unions: it is enough to check that this holds for $Lan_YG$, but this functor is lex cocontinuous.

    $\widehat{(-)}$ is a functor: Given $\alpha :G\to G'$, the family $(\alpha _{(x,p)}: G(x,p)\to G'(x,p))$ induces a map between the colimits.

    $\widetilde{(-)}\circ \widehat{(-)}\cong id$: First note that the cocone map $G(x,p_0)\to colim _{p\in S_{\mathcal{C}}(x)^{op}}G(x,p)$ is mono. Indeed, it is the $Lan_YG$-image of $Spec(\mathcal{C})(-,(x,p_0))\to Spec(\mathcal{C})(-,x)$, which is (pointwise) mono. 
    Hence the induced map in
\[\begin{tikzcd}
	{\bigcap _{u\in p_0} \widehat{G}(u)= \bigcap _{u\in p_0} colim _{p\in [u]^{op}}G(x,p)} && {\widehat{G}(x)=colim _{p\in S_{\mathcal{C}}(x)^{op}}G(x,p)} \\
	\\
	&& {G(x,p_0)}
	\arrow[hook, from=1-1, to=1-3]
	\arrow[dashed, from=3-3, to=1-1]
	\arrow[hook, from=3-3, to=1-3]
\end{tikzcd}\]
    is mono. We need to prove that it is an epi. 
    
    Since $\bigcap _{p_0} \widehat{ G}(u)$ can be written as a union of primes, it is enough to prove that if a prime subobject $a\hookrightarrow \widehat{G}(x)$ lies in $\bigcap _{p_0}\widehat{G}(u)$ then it lies in $G(x,p_0)$. Let $P\subseteq S_{\mathcal{C}}(x)$ be the poset of those prime filters $p$ such that $a$ lies in $G(x,p)$. The legs of the colimit cocone are jointly epimorphic, therefore $a$ is contained in at least one of them (as it is $\bigvee $-indecomposable), so $P$ is nonempty. Given a chain $(p_0\leq \dots p_{\alpha }\leq \dots )_{\alpha <\lambda }$ in $P$, its union $p_{\lambda }$ also belongs to $P$ as $G$ was $p_{\infty }$-flat. So by Zorn's lemma there is a maximal element $\overline{p}$ in $P$. We claim that $\overline{p}$ is the largest element of $P$.

    For that it suffices to show that $P$ is directed. So assume that we have a commutative diagram

\[\begin{tikzcd}
	& {G(x,p)} \\
	a && {\widehat{G}(x)} \\
	& {G(x,p')}
	\arrow[hook, from=1-2, to=2-3]
	\arrow[dashed, hook, from=2-1, to=1-2]
	\arrow[hook, from=2-1, to=2-3]
	\arrow[dashed, hook, from=2-1, to=3-2]
	\arrow[hook, from=3-2, to=2-3]
\end{tikzcd}\]
We need to find $p''\geq p,p'$ with $a\subseteq G(x,p'')$. As $a$ is prime, it is enough to prove that the pullback (intersection) of $G(x,p)$ and $G(x,p')$ is covered by subobjects of the form $G(x,p'')$. As $Lan_YG$ is lex cocontinuous, it is enough to prove that the intersection of 
\[
Spec(\mathcal{C})(-,(x,p)) \hookrightarrow Spec(\mathcal{C})(-,x) \hookleftarrow Spec(\mathcal{C})(-,(x,p'))
\]
can be covered by subobjects of the form $Spec(\mathcal{C})(-,(x,p''))$ in $\mathbf{Set}^{Spec(\mathcal{C})^{op}}$. This is clear; if $[f]:(z,q)\to x$ is a $q$-germ which is both $p$ and $p'$-continuous, then $f_!q\geq p, p'$ and $[f]\in Spec(\mathcal{C})((z,q),(x,f_!q))$. 

So $\overline{p}$ is the top element of $P$. Take $u\in p_0$. By assumption $a\subseteq colim _{p\in [u]^{op}}G(x,p)$, and as $a$ is prime, $a\subseteq G(x,p')$ for some $p'\in [u]$. Hence $p'\leq \overline{p}$, and therefore $u\in \overline{p}$. We got that $p_0\leq \overline{p}$. As a result $a\subseteq G(x,\overline{p})\subseteq G(x,p_0)$ and this is what we wanted to show.

    $\widehat{(-)}\circ \widetilde{(-)}\cong id$: We need that the induced map
    \[
    colim _{p\in S_{\mathcal{C}}(x)^{op}} \bigcap _{u\in p}F(u) \to F(x) 
    \]
    is iso.

    It is mono, i.e.~$\bigcup _p \widetilde{F}(x,p)$ coincides with the colimit. Indeed, the union is effective, so it is the colimit of the diagram 
\[\begin{tikzcd}
	{\widetilde{F}(x,p_i)} & {\widetilde{F}(x,p_j)} & {\widetilde{F}(x,p_k)} & \dots \\
	{R_{i,j}} & {R_{i,k}} & \dots
	\arrow[from=2-1, to=1-1]
	\arrow[from=2-1, to=1-2]
	\arrow[from=2-2, to=1-1]
	\arrow[from=2-2, to=1-3]
\end{tikzcd}\]
where $R_{i,j}$ is the intersection of $\widetilde{F}(x,p_i)$ and $\widetilde{F}(x,p_j)$. It suffices to show that each $R_{i,j}$ can be covered by the subobjects $\widetilde{F}(x,\overline{p})$ for $\overline{p}\geq p_i,p_j$, as in this case the two colimits coincide. But $\widetilde{F}(x,p_i) \cap \widetilde{F}(x,p_j)= \bigcap _{u\in p_i \cup p_j} F(u)$ can be written as the union of primes, and given a prime $a\subseteq \bigcap _{u\in p_i \cup p_j} F(u)$ we have $a\subseteq \widetilde{F}(x,tp(a))$ and $tp(a)\geq p_i,p_j$.

Epi: $F(x)$ is the union of primes, and if $a\subseteq F(x)$ is such a subobject, then it is contained in $\widetilde{F}(x,tp(a))$.
\end{proof}

\begin{definition}
\label{Edef2}
    Let $E$ be the set of families
\[
\left[
\begin{tikzcd}
	{(x,p_0)} & {\dots } & {(x,p_i)} & {\dots } \\
	&& {(y,q)}
	\arrow["{[f]}"', from=1-1, to=2-3]
	\arrow["{[f]}", from=1-3, to=2-3]
\end{tikzcd}
\right]_{p_i:\ f_!p_i\geq q}
\]
 where $f:x\to y$ is an effective epimorphism. We write $\tau _E$ for the generated Grothendieck topology.   
\end{definition}

\begin{remark}
\label{Ebar2}
    Proposition \ref{Ebar} remains true, so this $\tau _E$ has the same explicit description as the previous one ($\tau _{E^{\neg}}$). The proof is identical except that in the verification of pullback-stability we shall write  $f_!g'_!p'=g_!f'_!p'\geq g_!q'\geq q$.
\end{remark}

\begin{proposition}
\label{strictlyeq}
    Assume that $((x_i,p_i)\xrightarrow{[f_i]} (y,q))_i$ is a cover in $\tau _E$. Then there is an $i_0$, such that $(f_{i_0})_!p_{i_0}=q$.
\end{proposition}

\begin{proof}
    This is satisfied by each $E$-family, in other terms $S_{\mathcal{C}}$ preserves effective epis (this was left as an exercise in Section 2). Also, it is satisfied by the isomorphisms. Given a composite cotree as in Proposition \ref{Ebar}, we define such a branch inductively. As there is no infinite branch, we will arrive to a leaf. The composite of this branch yields a leg $[f]:(x,p)\to (y,q)$ with $f_!p=q$.
\end{proof}

\begin{remark}
    Our goal is to study $Sh(Spec(\mathcal{C}),\tau _E)$. In \cite{joyaltoptp} the "topos of existential types" is defined as $Sh(Spec(\mathcal{C}),\tau _{E_0})$ where $E_0$ is the same as $E$ from Definition \ref{Edef2}, except that they only keep those legs $[f]:(x,p_i)\to (y,q)$ which satisfy the strict equality $f_!p_i =q$. It is unclear to us, whether this set is "pullback-stable" (the easy argument of Remark \ref{Ebar2} fails, as $f'_!p'=q'$ only implies $f_!g'_!p'=g_!f'_!p'=g_!q'\geq q$, while in this case we would need a strict equality). This indicates that our topology might be more natural, hence we proceed with that.
\end{remark}

\begin{proposition}
\label{restrtoEpres}
    $\mathcal{C}$ is coherent, $\mathcal{E}$ is a prime-generated Grothendieck topos. The equivalence of Proposition \ref{lexvequiv} restricts between the full subcategories $E-Flat_{p_{\infty }}(Spec(\mathcal{C}),\mathcal{E})$ 
(of $p_{\infty }$-flat, $E$-preserving functors) and $\mathbf{Coh}(\mathcal{C},\mathcal{E})$.
\end{proposition}

\begin{proof}
    If $F$ preserves effective epis then $\widetilde{F}$ is $E$-preserving: $F(f)^{-1}(\widetilde{F}(y,q))$ can be written as a union of primes $\bigcup _i a_i$, and $f_!(tp(a_i))\geq q$, as $v\in q$ implies $a\subseteq F(f)^{-1}(Fv)=F(f^{-1}v)$. So the composite family $(\widetilde{F}(x,tp(a_i))\hookrightarrow F(x) \xtwoheadrightarrow{F(f)} F(y))_i$ covers $\widetilde{F}(y,q)$.

    If $G$ is $E$-preserving then $\widehat{G}$ preserves effective epis: We need that the upper horizontal arrow in

\[
\adjustbox{scale=0.8}{
\begin{tikzcd}
	& {colim_{p\in S_{\mathcal{C}}(x)^{op}}G(x,p)} && {colim_{q\in S_{\mathcal{C}}(y)^{op}}G(y,q)} \\
	\\
	& {G(x,p_i)} && {G(y,q_0)} \\
	{G(x,p_j)}
	\arrow[from=1-2, to=1-4]
	\arrow[hook, from=3-2, to=1-2]
	\arrow[dashed, from=3-2, to=3-4]
	\arrow[hook, from=3-4, to=1-4]
	\arrow[hook, from=4-1, to=1-2]
	\arrow[dashed, from=4-1, to=3-4]
\end{tikzcd}
}
\]
    is effective epi. By assumption each $G(y,q_0)$ is covered by the dashed arrows (for $f_!(p_i)\geq q_0$). The composite tree is covering $colim _q G(y,q)$, and each leg factors through the upper horizontal arrow, which is therefore effective epimorphic.
\end{proof}

\begin{theorem}
\label{main2}
     Let $\mathcal{C}$ be a coherent category. Write $\varphi :\mathcal{C}\to Sh(Spec(\mathcal{C}),\tau _E)$ for the functor taking $y$ to $Spec(\mathcal{C})(-,y)^{\#}$ (sheafification of $\varphi _0(y)$). Then:
    \begin{itemize}
        \item[i)] $\varphi $ is coherent.
        \item[ii)] If $\mathcal{E}$ is a prime-generated Grothendieck topos, then 
        \[ \varphi ^{\circ}:Fun^*_{p_{\infty}}(Sh(Spec(\mathcal{C})),\mathcal{E})\to \mathbf{Coh}(\mathcal{C},\mathcal{E})\]
        is an equivalence. Here $Fun^*_{p_{\infty }}$ denotes lex cocontinuous functors, with the additional property, that whenever $(p_0\leq \dots p_{\alpha }\leq \dots )_{\alpha <\lambda }$ is a chain of prime-filters with union $p_{\lambda }$, the cone $\#Y(x,p_0)\hookleftarrow \dots \#Y(x,p_{\alpha })\hookleftarrow \dots \#Y(x,p_{\lambda})$ is taken to a limit (intersection) in $\mathcal{E}$.
    \end{itemize}
\end{theorem}

\begin{proof}
$i)$ We already know that $\varphi =\# \circ \varphi _0$ preserves finite limits and finite unions. To see that it preserves effective epis apply the second half of the proof of Proposition \ref{restrtoEpres} with $G=\#Y$ (here we didn't use that $\mathcal{E}$ was prime-generated).

$ii)$  The equivalence $(\# Y)^{\circ}:Fun^*(Sh(Spec(\mathcal{C})),\mathcal{E})\to E-Flat(Spec(\mathcal{C})),\mathcal{E})$ restricts between $Fun^*_{p_{\infty }}(Sh(Spec(\mathcal{C})),\mathcal{E})$ and $E-Flat_{p_{\infty}}(Spec(\mathcal{C}),\mathcal{E})$. Hence
\[
Fun^*_{p_{\infty }}(Sh(Spec(\mathcal{C})),\mathcal{E})\xrightarrow{(\#Y)^{\circ }} E-Flat_{p_{\infty}}(Spec(\mathcal{C}),\mathcal{E}) \xrightarrow{\widehat{(-)}} \mathbf{Coh}(\mathcal{C},\mathcal{E})
\]
    is an equivalence. It remains to check that it is isomorphic to $\varphi ^{\circ }$. 

Take $M^*:Sh(Spec(\mathcal{C}))\to \mathcal{E}$ cocontinuous, $p_{\infty }$-lex. We have
\[
\adjustbox{width=\textwidth}{
\begin{tikzcd}
	{colim _{p\in S_{\mathcal{C}}(x)^{op}}(M^*\circ \# \circ Y)(x,p)} && {M^*\circ \#\ (colim _{p} Y(x,p))} & {M^*\circ \# \ (Spec(\mathcal{C})(-,x))} \\
	\\
	{colim _{q\in S_{\mathcal{C}}(y)^{op}}(M^*\circ \# \circ Y)(y,q)} && {M^*\circ \# \ (colim _{q} Y(y,q))} & {M^*\circ \# \ (Spec(\mathcal{C})(-,y))}
	\arrow["{\cong }", from=1-1, to=1-3]
	\arrow["{\left(M^*\circ \# \circ Y \ ([f]) \right)_p}", from=1-1, to=3-1]
	\arrow[equals, from=1-3, to=1-4]
	\arrow[from=1-3, to=3-3]
	\arrow["{M^*\circ \# \ (f_{\circ })}", from=1-4, to=3-4]
	\arrow["{\cong }", from=3-1, to=3-3]
	\arrow[equals, from=3-3, to=3-4]
\end{tikzcd}
}
\]
which is natural in $M$.

\end{proof}

\begin{remark}
    The induced map $\varphi ^*: Sh(\mathcal{C})\leftrightarrow Sh(Spec(\mathcal{C})):\varphi _*$ is a geometric surjection (i.e.~$\varphi ^*$ is conservative). The argument is identical to the one in Remark \ref{varphigeomsurj}.
\end{remark}

\section{Comparing three cases}

\begin{proposition}
\label{compatible}
   $\mathcal{C}$ is a coherent category. Let $(x,p)\xrightarrow{[f]}(y,q)\xleftarrow{[f']}(x',p')$ be a cospan in $Spec(\mathcal{C})$. Then it can be completed to a commutative square
\[
\adjustbox{scale=0.9}{
\begin{tikzcd}
	{(x'',p'')} & {(x',p')} \\
	{(x,p)} & {(y,q)}
	\arrow["{[g']}", from=1-1, to=1-2]
	\arrow["{[g]}"', from=1-1, to=2-1]
	\arrow["{[f']}", from=1-2, to=2-2]
	\arrow["{[f]}"', from=2-1, to=2-2]
\end{tikzcd}
}
\]
iff $f_!p$ and $f'_!p'$ are compatible (i.e.~their union has f.i.p.).
\end{proposition}

\begin{proof}
    $\Rightarrow $ is clear: $f'_!p'\leq f'_!(g'_!p'')=f_!(g_!p'')\geq f_!p$.

    $\Leftarrow $: Assume that $f:x\supseteq u_0\to y$ and $f':x'\supseteq u'_0\to y$ with $u_0\in p$, $u'_0\in p'$. First we show that $(u_0,p)\xrightarrow{[f]}(y,q)\xleftarrow{[f']}(u_0',p')$ can be completed to a commutative square ($p$ stands for $\restr{p}{u_0}$).

    We claim that there is $p''$ making $\pi _1$ and $\pi _2$ continuous in 
\[
\adjustbox{scale=0.9}{
\begin{tikzcd}
	{(u_0\times _yu_0',p'')} & {(u_0',p')} \\
	{(u_0,p)} & {(y,q)}
	\arrow["{[\pi _2]}", from=1-1, to=1-2]
	\arrow["{[\pi _1]}"', from=1-1, to=2-1]
	\arrow["{[f']}", from=1-2, to=2-2]
	\arrow["{[f]}"', from=2-1, to=2-2]
\end{tikzcd}
}
\]
That is, we need that $u_0\supseteq u\in p$, $u_0'\supseteq u'\in p'$ implies $\pi _1^{-1}(u)\cap \pi _2^{-1}(u')\neq \emptyset $, equivalently: $\exists _f \exists _{\pi _1}(\pi _1^{-1}(u)\cap \pi _2^{-1}(u'))\neq \emptyset$. But the latter expression can be written as
    \[
        \exists _f \exists _{\pi _1}(\pi _1^{-1}(u)\cap \pi _2^{-1}(u'))  = \exists _f(u\cap \exists _{\pi _1}(\pi _2^{-1}u'))
        = \exists _f(u\cap f^{-1}(\exists _{f'}u'))=\exists _fu \cap \exists _{f'}u'
    \]
    (see \cite[Proposition 39]{GARNER2020102831}). This is nonzero as $f_!p$ and $f'_!p'$ are compatible, and $f^{-1}(\exists _fu)\geq u \in p$, $f'^{-1}(\exists _{f'}u')\geq u' \in p'$, so $\exists _fu \in f_!p$ and $\exists _{f'}u'\in f'_!p'$.

    The original cospan is completed to a square as
\[
\adjustbox{scale=0.9}{
\begin{tikzcd}
	{(u_0\times _yu_0',p'')} & {(u_0',p')} & {(x',p')} \\
	{(u_0,p)} \\
	{(x,p)} && {(y,q)}
	\arrow["{[\pi _2]}", from=1-1, to=1-2]
	\arrow["{[\pi _1]}"', from=1-1, to=2-1]
	\arrow["{[id]}", from=1-2, to=1-3]
	\arrow["{[f']}"{description}, from=1-2, to=3-3]
	\arrow["{[f']}", from=1-3, to=3-3]
	\arrow["{[id]}"', from=2-1, to=3-1]
	\arrow["{[f]}"{description}, from=2-1, to=3-3]
	\arrow["{[f]}"', from=3-1, to=3-3]
\end{tikzcd}
}
\]
\end{proof}
\vspace{-5mm}

\begin{theorem}
\label{weaklylc}
    Let $\mathcal{C}$ be a coherent category. Then $Sh(Spec(\mathcal{C}),\tau _E)$ is locally connected.
\end{theorem}

\begin{proof}
    As the epimorphic image of a connected object is connected, it is enough to prove that $Sh(Spec(\mathcal{C}))$ has a generating set of connected objects. We claim that each sheafified representable $\#Y(y,q)=Spec(\mathcal{C})(-,(y,q))^{\#}$ is connected.

    First note that $\#Y(y,q)\neq \emptyset$. Indeed, if the empty family was covering $\#Y(y,q)$, then by \cite[Proposition 1.3.3(2)]{makkai} the empty family on $(y,q)$ would belong to $\tau _E$, which is not the case.

    Write it as a disjoint union: $\#Y(y,q)= A\sqcup B$. By \cite[Lemma 6.1.2]{makkai} we find a cover $(\#Y(x_i,p_i)\to  A)_i$ and a cover $(\#Y(x'_j,p'_j)\to  B)_j$, such that the composites $\#Y(x_i,p_i)\to \#Y(y,q)$ are of the form $\#Y([f_i])$ and the composites $\#Y(x'_j,p'_j)\to \#Y(y,q)$ are of the form $\#Y([f'_j])$.

    The union of these two families covers $\#Y(y,q)$. By \cite[Proposition 1.3.3(2)]{makkai} again, $((x_i,p_i)\xrightarrow{[f_i]} (y,q) \xleftarrow{[f'_j]} (x_j',p'_j))_{i,j}$ belongs to $\tau _E$. Therefore, by Proposition \ref{strictlyeq}, there is a leg, say $(x_0,p_0)\xrightarrow{[f_0]}(y,q)$, with $(f_0)_!p_0=q$. In particular $(f_0)_!p_0$ is compatible with $(f'_j)_!p'_j$ for any $j\in J$, hence by Proposition \ref{compatible} there is a square
\[\begin{tikzcd}
	{(x_0,p_0)} & {(y,q)} \\
	{(x'',p'')} & {(x_j',p_j')}
	\arrow["{[f_0]}", from=1-1, to=1-2]
	\arrow["{[g]}", from=2-1, to=1-1]
	\arrow["{[g']}"', from=2-1, to=2-2]
	\arrow["{[f'_j]}"', from=2-2, to=1-2]
\end{tikzcd}\]
    But $\#Y(x'',p'')\neq \emptyset$, which contradicts that $A$ and $B$ were disjoint. So we get $J=\emptyset $ and $B=\emptyset$.
\end{proof}

\begin{remark}
\label{makkaiversion}
    $\mathcal{C}$ is a coherent category. We briefly relate a third version to the previous two, namely Makkai's topos of types from \cite{10.1007/BFb0090947}. It is defined as $Sh(Spec(\mathcal{C}),\tau _{sat})$ where $\tau _{sat}$ is the topology generated by the singleton families $(x,p)\xrightarrow{[f]}(y,f_!p)$ ($[f]$ is arbitrary). Proposition \ref{compatible} implies that $id:(Spec(\mathcal{C}),\tau _E)\to (Spec(\mathcal{C}), \tau _{sat})$ is cover-preserving (and its post-composition with $\#Y$ is of course flat), so we get a geometric embedding $Sh(Spec(\mathcal{C}), \tau _{sat})\to Sh(Spec(\mathcal{C}), \tau _E)$.

    By \cite[Theorem 1.1 and (1.7)]{10.1007/BFb0090947} the topos $Sh(Spec(\mathcal{C}),\tau _{sat})$ is prime-generated, each representable $Spec(\mathcal{C})(-,(x,p))$ as well as $Spec(\mathcal{C})(-,x)$ is a sheaf wrt.~$\tau _{sat}$, and $\varphi ^s$ sending $x$ to $Spec(\mathcal{C})(-,x)=Spec(\mathcal{C})(-,x)^{\#}$ is a coherent functor. By Theorem \ref{main2} we get a $p_{\infty}$-lex cocontinuous functor $\widetilde{(\varphi ^s)}^*: Sh(Spec(\mathcal{C}),\tau _E)\to Sh(Spec(\mathcal{C}),\tau _{sat})$ whose restriction to $Spec(\mathcal{C})$ is given by $(x,p)\mapsto \bigcap _{u\in p} Spec(\mathcal{C})(-,u) = Spec(\mathcal{C})(-,(x,p))$, so it coincides with the inverse image part of the geometric embedding from the previous paragraph.
\end{remark}

We glue all of our previous results into a single diagram:

\begin{theorem}
    \label{allthree}
Let $\mathcal{C}$ be a coherent category with finite disjoint coproducts and take any small (non-full) subcategory $K\subseteq \mathbf{Coh}(\mathcal{C},\mathbf{Set})$. Then there is a commutative diagram:
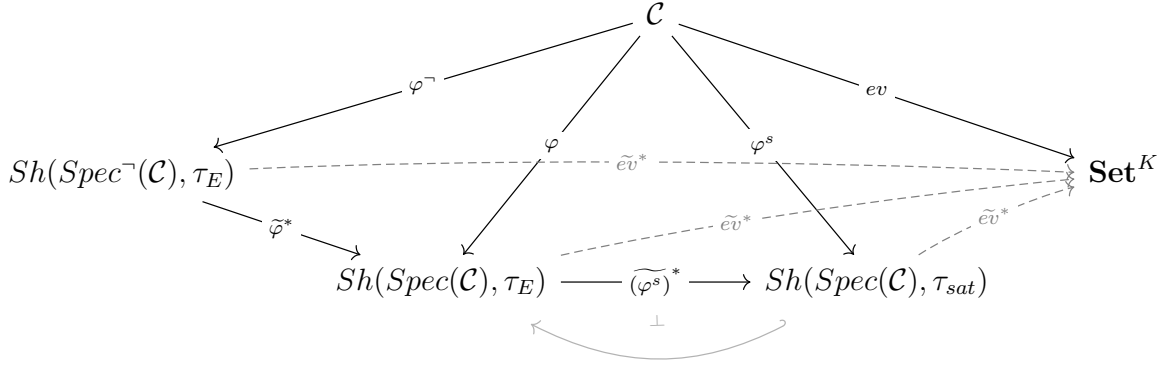
\[
\adjustbox{width=\textwidth}{
\begin{tikzcd}
	&& {\mathcal{C}} \\
	\\
	{Sh(Spec^{\neg}(\mathcal{C}),\tau _E)} &&&& {\mathbf{Set}^K} \\
	& {Sh(Spec(\mathcal{C}),\tau _{E})} && {Sh(Spec(\mathcal{C}),\tau _{sat})}
	\arrow["{\varphi ^{\neg}}"{description}, from=1-3, to=3-1]
	\arrow["ev"{description}, from=1-3, to=3-5]
	\arrow["\varphi"{description}, from=1-3, to=4-2]
	\arrow["{\varphi ^s}"{description}, from=1-3, to=4-4]
	\arrow["{\widetilde{ev}^*}"{description}, color={rgb,255:red,128;green,128;blue,128}, curve={height=-6pt}, dashed, from=3-1, to=3-5]
	\arrow["{\widetilde{\varphi }^*}"{description}, from=3-1, to=4-2]
	\arrow["{\widetilde{ev}^*}"{description, pos=0.4}, color={rgb,255:red,128;green,128;blue,128}, curve={height=-6pt}, dashed, from=4-2, to=3-5]
	\arrow[""{name=0, anchor=center, inner sep=0}, "{\widetilde{(\varphi ^s)}^*}"{description}, from=4-2, to=4-4]
	\arrow["{\widetilde{ev}^*}"{description}, color={rgb,255:red,128;green,128;blue,128}, curve={height=-6pt}, dashed, from=4-4, to=3-5]
	\arrow[""{name=1, anchor=center, inner sep=0}, shift left=2, color={rgb,255:red,179;green,179;blue,179}, curve={height=-18pt}, hook', from=4-4, to=4-2]
	\arrow["\dashv"{anchor=center, rotate=-88}, color={rgb,255:red,179;green,179;blue,179}, draw=none, from=0, to=1]
\end{tikzcd}
}
\]
The four maps out of $\mathcal{C}$ are coherent, the $\varphi $'s are conservative, and the maps with superscript $*$ are lex cocontinuous. 

(If we erase $Sh(Spec^{\neg }(\mathcal{C}),\tau _E)$ then $\mathcal{C}$ can be any coherent category.)
\end{theorem}

\begin{proof}
    We get $\widetilde{\varphi }^*$ by Theorem \ref{main1} as $Sh(Spec(\mathcal{C}),\tau _E)$ is locally connected by Theorem \ref{weaklylc}. The map $(\widetilde{\varphi ^s})^*$ is described in Remark \ref{makkaiversion}. The extensions $\widetilde{ev}^*$ exist by Theorem \ref{main1}, Theorem \ref{main2} and \cite[Theorem 1.1]{10.1007/BFb0090947}, as $\mathbf{Set}^K$ is prime-generated (in particular it is locally connected). Triangles with two dashed faces commute, by the uniqueness of lex cocontinuous/ $p_{\infty }$-lex cocontinuous extensions. 
\end{proof}

\begin{proposition}
Let $\mathcal{C}$ be a countable coherent category, such that for any $x\in \mathcal{C}$ the set $S_{\mathcal{C}}(x)=Spec(Sub_{\mathcal{C}}(x))$ is countable. 
\begin{itemize}
    \item[i)] If $\mathcal{C}$ has finite disjoint coproducts then $Sh(Spec^{\neg }(\mathcal{C}),\tau _E)$ has enough points and $\widetilde{\varphi }^*$ is conservative.
    \item[ii)] $Sh(Spec(\mathcal{C}),\tau _E)$ has enough points.
\end{itemize}
\end{proposition}

\begin{proof}
    By the cardinality assumption, the sites $(Spec^{\neg }(\mathcal{C}),E)$ and $(Spec(\mathcal{C}),E)$ are separable. Hence by \cite[Theorem 6.2.4]{makkai} $Sh(Spec^{\neg }(\mathcal{C}),\tau _E)$ and $Sh(Spec(\mathcal{C}),\tau _E)$ have enough points. 

    Let $\langle M_i^*\rangle _i:Sh(Spec^{\neg }(\mathcal{C}))\to \mathbf{Set}^I$ be conservative, lex, cocontinuous. Then we get a lex cocontinuous extension $\overline{M}$ in
\[\begin{tikzcd}
	&& {Sh(Spec(\mathcal{C}),\tau _E)} \\
	{\mathcal{C}} &&&& {\mathbf{Set}^I} \\
	&& {Sh(Spec^{\neg}(\mathcal{C}),\tau _E)}
	\arrow["{\overline{M}}", from=1-3, to=2-5]
	\arrow[""{name=0, anchor=center, inner sep=0}, "{\varphi }", from=2-1, to=1-3]
	\arrow[""{name=1, anchor=center, inner sep=0}, "{\varphi ^{\neg}}"', from=2-1, to=3-3]
	\arrow["{\widetilde{\varphi }^*}"', from=3-3, to=1-3]
	\arrow["{\langle M_i^*\rangle _i}"', from=3-3, to=2-5]
	\arrow["{=}", draw=none, from=0, to=1]
\end{tikzcd}\]
    such that the outer square commutes. By the uniqueness of lex cocontinuous extensions the right hand triangle commutes. It follows that $\widetilde{\varphi }^*$ is conservative. 
\end{proof}

\begin{question}
    Is $\widetilde{\varphi }^*$ conservative when $\mathcal{C}$ is an arbitrary coherent category with finite disjoint coproducts?
\end{question}

\printbibliography

\end{document}